\documentclass[a4paper,11pt]{amsart}

\usepackage{amsmath}
\usepackage{amssymb}
\usepackage{amsthm}
\usepackage{verbatim}
\usepackage[all]{xy}
\usepackage{enumerate}
\usepackage{color}

\newtheorem{thm}{Theorem}[section]

\newtheorem{lem}[thm]{Lemma}

\theoremstyle{definition}

\newtheorem{rmk}[thm]{Remark}

\numberwithin{equation}{section}
\title{Weak type estimates for the absolute value mapping}

\author{M. Caspers, D. Potapov, F. Sukochev, D. Zanin}
\date{\today, {\it MSC2000}: 47B10, 47L20, 47A30,\\ {\it Acknowledgement:} The first author is supported by the grant SFB 878}

\address{M. Caspers, Fachbereich Mathematik und Informatik der Universit\"at M\"unster,
Einsteinstrasse 62,
48149 M\"unster, Germany}
\email{martijn.caspers@uni-muenster.de}
\address{D. Potapov, F. Sukochev, D. Zanin, School of Mathematics and Statistics, UNSW, Kensington 2052, NSW, Australia}
\email{d.potapov@unsw.edu.au}
 \email{f.sukochev@unsw.edu.au}
\email{d.zanin@unsw.edu.au}

\begin{document}

\maketitle

\begin{abstract} We prove that if $A$ and $B$ are bounded self-adjoint operators such that $A -B$ belongs to the trace class, then
$|A| -|B|$ belongs to the principal ideal $\mathcal{L}_{1,\infty}$ in the algebra $\mathcal{L}(H)$ of all
 bounded operators on an infinite-dimensional Hilbert space generated by an operator whose sequence of eigenvalues is
$\{1, \frac12,\frac13,\dots\}$. Moreover, $\mu(j;|A| -|B|)\leq const(1 + j )^{-1}\|A-B\|_1$. We also obtain a semifinite version of this result, as well as the corresponding commutator estimates.
\end{abstract}

\bibliographystyle{plain}

\parindent=0.0pt

\section{Introduction}
Let $H$ be a complex separable Hilbert space, let $\mathcal{K}(H)$ be the $*-$algebra of all compact operators on $H$ and let $\mathcal{L}_p$, $1\leq p<\infty,$ be the $p$-th Schatten-von Neumann class (that is the class of all  operators $A$ from $\mathcal{K}(H)$ such that $\|A\|_p:=\left(\sum _{k=0}^\infty \mu(k; A)^p\right)^{1/p}<\infty$, where $\{\mu(k;A)\}_{k=0}^\infty$ is the sequence of singular numbers of the operator $A$ \cite{GK1,LSZ}). The following result was proved by E.~B. Davies \cite[Theorem 8]{Davies} (for its extension to semifinite von Neumann algebras, we refer to \cite{DDPS1}).
\begin{thm}\label{Daviesthm} If $A,B$ are self-adjoint bounded operators on $H$ and if $A-B\in\mathcal{L}_p,$ $1<p<\infty$, then $|A|-|B|\in\mathcal{L}_p$ and
$$\||A|-|B|\|_p\leq c_p\|A-B\|_p.$$
Here, $c_p$ depends only on $p$ and $c_p=O(p)$ as $p\to\infty$ and $c_p=O((p-1)^{-1})$ as $p\to 1$.
\end{thm}

For various extensions and generalizations of Theorem \ref{Daviesthm}, we refer to the papers \cite{Kos}, \cite{BiS}, \cite{DDPS1}, \cite{DDPS2}, \cite{NP} studying the Lipschitz continuity of the absolute value mapping $A\to |A|$ in the setting of symmetrically-normed ideals (and more general symmetric operator spaces). Here, we contribute to an interesting open question concerning the optimal form of Theorem \ref{Daviesthm} in the crucial case $p=1$. It is well known (see \cite[Section 3]{Davies}) that the absolute value mapping is not Lipschitz continuous in the
trace class $({\mathcal L}_1,\|\cdot\|_1)$.  It was proved by H. Kosaki \cite[Theorem 12]{Kos} (see also \cite[Corollary 3.4]{DDPS2}) that the absolute value mapping is Lipschitz continuous from
$({\mathcal L}_1,\|\cdot\|_1)$ into Banach ideal $({\mathcal M}_{1,\infty},\|\cdot\|_{{\mathcal M}_{1,\infty}})$, where
$$
{\mathcal M}_{1,\infty}:=\{A\in \mathcal{K}(H): \|A\|_{{\mathcal M}_{1,\infty}}:=\sup_{N\ge 0} \frac{1}{\log(N+2)}\sum _{k=0}^N \mu(k;A)<\infty\}.
$$
The main objective of this paper is to show that the latter result holds if we replace  $({\mathcal M}_{1,\infty},\|\cdot\|_{{\mathcal M}_{1,\infty}})$ with a smaller (quasi-Banach) ideal $({\mathcal L}_{1,\infty},\|\cdot\|_{1,\infty})$, where
\begin{equation}\label{weakL_1}
{\mathcal L}_{1,\infty}:=\{A\in \mathcal{K}(H): \|A\|_{{\mathcal L}_{1,\infty}}:=\sup_{k\ge 0} (k+1) \mu(k;A)<\infty\}.
\end{equation}

\begin{thm}\label{abs lip} If $A,B$ are self-adjoint bounded operators on $H$ and if $A-B\in\mathcal{L}_1$, then $|A|-|B|\in\mathcal{L}_{1,\infty}$ and
\begin{equation}\label{super}\||A|-|B|\|_{1,\infty}\leq \Big(34+\frac{2560e}{\pi}\Big)\|A-B\|_1.
\end{equation}

\end{thm}

The strength of Theorem \ref{abs lip} is seen from the fact that it implies the result of Theorem \ref{Daviesthm} via a combination of methods used in \cite{DDPS1}, \cite{Davies}
linking Lipschitz continuity and commutator estimates with a noncommutative version of the Boyd interpolation theorem
(see e.g. \cite[Theorem 5.8]{Dirksen}).   We refer to Remark \ref{Rmk=WeakInterpolation} for more details.   Such an implication is of course not available from the results of
\cite[Theorem 12]{Kos} and \cite[Corollary 3.4]{DDPS2}. The result of Theorem \ref{abs lip} is also sharp in the sense that the quasi-norm
$\|\cdot\|_{1,\infty}$ is the largest symmetric quasi-norm on the ideal of finite rank operators for which \eqref{super} holds
(the latter follows from the proof of \cite[Lemma 10]{Davies}).

From a certain perspective, the result of Theorem \ref{abs lip} is not unexpected. Indeed, the proof of Theorem \ref{Daviesthm} in \cite {Davies},
as well as the proofs of its analogues and extensions from \cite{Kos}, \cite{BiS}, \cite{DDPS1}, \cite{DDPS2} are ultimately based on the famous
results due to V.I. Macaev, I. C. Gohberg and M. G. Krein (see \cite{GK1}, \cite{Davidson}), describing the behavior of (generalized) triangular
truncation operators in Schatten-von Neumann classes $\mathcal{L}_p$. In the case when $p=1$, these results yield the fact that the latter operator
acts boundedly from the Banach space $(\mathcal{L}_1,\|\cdot\|_1)$ into a quasi-Banach space $(\mathcal{L}_{1,\infty},\|\cdot\|_{1,\infty})$.
However, (and here lies the major difficulty) all the proofs in the just listed papers involve certain integration processes, which render them
inapplicable in the quasi-normed setting. Exactly the same obstacle also manifested itself in \cite[Theorem 2.5(i)]{NP}. Indeed, that theorem
yields the result of Theorem \ref{abs lip} under the restrictive assumption that $\rm{rank}(A-B)=1$ and the methods used in \cite{NP} do not
seem applicable to treat the general case. To circumvent this difficulty, we employ a completely different approach coming back to a
celebrated theorem of I. Schur concerning positive semidefiniteness of a Schur (or Hadamard) product of two semidefinite matrices.

 Section \ref{Sect=MainTheorem} contains the proof of Theorem \ref{abs lip}. In Section \ref{Sect=CompactOperators} we find a sharper result assuming the extra condition that $A$ and $B$ are compact operators. In this case it turns out that $\vert A \vert - \vert B \vert$ in fact lands in the separable part of $\mathcal{L}_{1, \infty}$, see Theorem \ref{Thm=SeperablePart}. Section \ref{Sect=SemiFinite} contains the extension of Theorem \ref{abs lip} to the setting of semifinite von Neumann algebras. This theme has been explored already in \cite{DDPS2},
however, methods employed there (again due to the obstacle explained above) were not sufficiently strong to obtain the weak type estimate similar to \eqref{super}. Furthermore, the setting used in \cite{DDPS2} was restricted to the case of semifinite factors. The approach used in this paper allows us to dispense with the latter condition.  In Section \ref{Sect=CommutatorEstimates} we treat the consequences of Theorem \ref{abs lip} for commutator estimates. In the final Section we give a treatment of the consequences of  Theorem \ref{abs lip} for certain Lipschitz functions $f$ belonging to a subclass of the  Davies class (the case $f =  \vert \: \cdot \: \vert$ being Theorem \ref{abs lip}). Note that for general Lipschitz functions $f$ outside of that subclass the question whether the weak $(1, 1)$  estimate holds remains open.

%

\subsubsection*{Acknowledgement} The authors wish to express their gratitude to Peter Dodds and the anonymous referee for several improvements of this paper.

\section{Preliminaries}

\subsection{Singular values}
Let $\mathcal{L}(H)$ be the $*-$algebra of all bounded operators on the Hilbert space $H$ equipped with a uniform norm $\|\cdot\|_{\infty}$.
Every proper ideal in $\mathcal{L}(H)$ consists of compact operators. For brevity, we set $\mu(A):= \{\mu(k,A)\}_{k\geq0}$.
If $B\in\mathcal{L}(H)$ and $A\in\mathcal{K}(H)$, then it is well known that
\begin{equation}\label{msing}
\mu(AB)\leq\|B\|_{\infty}\mu(A),\quad \mu(BA)\leq\|B\|_{\infty}\mu(A),\quad \mu(A^*)=\mu(A).
\end{equation}
If  $A,B\in\mathcal{K}(H)$, then (see e.g. Corollary 2.3.16 in \cite{LSZ}) we have
\begin{equation}\label{asing}
\mu(A+B)\leq\sigma_2(\mu(A)+\mu(B)).
\end{equation}
Here, the dilation operator $\sigma_2:l_{\infty}\to l_{\infty}$ (acting on the space $l_\infty$ of all complex bounded sequences) is defined as follows
$$\sigma_2(a_0,a_1,\cdots)=(a_0,a_0,a_1,a_1,\cdots).$$

Two self-adjoint operators $A,B\in\mathcal{K}(H)$ are called {\it identically distributed} if $\mu(A_+)=\mu(B_+)$ and $\mu(A_-)=\mu(B_-)$. Here, $A=A_+ - A_-$ is the orthogonal decomposition of a self-adjoint operator
$A$ (see e.g. \cite[p.36]{BR}). A self-adjoint operator $A\in\mathcal{K}(H)$ is called {\it symmetrically distributed} if $\mu(A_+)=\mu(A_-).$

In what follows, the symbol ${\rm supp}(A)$ stands for the support projection of a self-adjoint operator $A\in\mathcal{L}(H)$ (that is, the spectral projection of $A$ corresponding to the set $\mathbb{R}\backslash\{0\}$).

\subsection{Ideal $\mathcal{L}_{1,\infty}$}
%

Let $A_0\in\mathcal{K}(H)$ be such that $\mu(A_0)=\{1/(k+1)\}_{k\geq0}$. The principal ideal generated by $A_0$ is frequently called weak-$\mathcal{L}_1$  and coincides with $\mathcal{L}_{1,\infty}$.
The mapping $\|\cdot\|_{1,\infty}$ on $\mathcal{L}_{1,\infty}$ (see \eqref{weakL_1})
is a quasi-norm. Indeed, it follows from \eqref{asing} that
\begin{equation}\label{quasi-triangle}
\|A+B\|_{1,\infty}\leq 2\|A\|_{1,\infty}+2\|B\|_{1,\infty}.
\end{equation}
It can be shown (see e.g \cite[Theorem 2.11.32]{Pietsch}) that the quasi-normed space $(\mathcal{L}_{1,\infty},\|\cdot\|_{1,\infty})$ is complete and is, therefore, a quasi-Banach ideal. It is important to note that the quasi-norm $\|\cdot\|_{1,\infty}$ is not equivalent to any norm. In particular, the weak form of triangle inequality in \eqref{quasi-triangle} is the best possible.

The ideals $\mathcal{L}_p$ and $\mathcal{L}_{1,\infty}$ both have the Fatou property. That is, if $A_n\in\mathcal{L}_{1,\infty},$ $\|A_n\|_{1,\infty}\leq1$ and $A_n\to A$ in measure, then $A\in\mathcal{L}_{1,\infty}$ and $\|A\|_{1,\infty}\leq 1.$ Exactly the same assertion holds for $\mathcal{L}_p.$

\subsection{Schur multiplication}\label{Sect=SchurMultiplication}

Let
$$M_n=\{A=\{A_{k,l}\}_{k,l=0}^{n-1}\}$$
be the $^*-$algebra of all complex $n\times n$ matrices. The algebra $M_n$ is $^*-$isomorphic to a subalgebra $P_n\mathcal{L}(H)P_n$ in $\mathcal{L}(H),$ where $P_n$ is a projection in $\mathcal{L}(H)$ such that
${\rm Tr}(P_n)=n$, where ${\rm Tr}$ is the standard trace on $\mathcal{L}(H)$. We also frequently identify $M_n$ with $\mathcal{L}(H)$ when ${\rm dim}(H)=n$. In the latter case, all previously introduced notations  (e.g.
$\mu(A), \|\cdot\|_p, \|\cdot\|_{1,\infty}, {\rm Tr},{\rm supp}(A)$) and terminology (e.g. identically distributed) remain unambiguously defined.

For every $A,B\in M_n,$ we define their Schur product (also called Hadamard product) $A\circ B$ by setting
$$(A\circ B)_{k,l}=A_{k,l}B_{k,l},\quad 0\leq k,l\leq n-1.$$
Fix $B\in M_n.$ The Schur multiplication operator $M_B:M_n\to M_n$ is defined by setting $M_B(A)=A\circ B.$ If $B\geq 0,$ then, according to the Schur theorem, we have that $M_B(A)\geq0$ for every $A\geq0.$ For a beautiful exposition of the latter theorem and other relevant properties of Schur multiplication, we refer the reader to \cite{Bhatia}. For the next lemma, see \cite{AndoHornJohnson}. We included a short proof for completeness.

\begin{lem}\label{schur positive} If $B\geq0,$ then
$$\|M_B(A)\|_1\leq 4\|{\rm diag}(B)\|_{\infty}\cdot\|A\|_1,\quad A\in M_n.$$
\end{lem}
\begin{proof} Suppose first that $A\geq 0.$ It follows that $M_B(A)\geq 0.$ Hence,
\begin{align*}\|M_B(A)\|_1&={\rm Tr}(M_B(A))={\rm Tr}(A\circ B)=\sum_{k=0}^{n-1}A_{k,k}B_{k,k}\cr &
\leq(\max_{0\leq k<n}B_{k,k})\cdot(\sum_{k=0}^{n-1}A_{k,k})=\|{\rm diag}(B)\|_{\infty}{\rm Tr}(A)=\|{\rm diag}(B)\|_{\infty}\|A\|_1.
\end{align*}

Consider now the general case of an arbitrary $A\in M_n.$ Using Jordan decomposition (see e.g. \cite[p.216]{BR}), we write
$$A=A_1-A_2+iA_3-iA_4,\quad A_m\geq0,\ \|A_m\|_1\leq\|A\|_1,\quad 1\leq m\leq 4.$$
Therefore,
$$\|M_B(A)\|_1\leq\sum_{m=1}^4\|M_B(A_m)\|_1\leq\|{\rm diag}(B)\|_{\infty}\cdot(\sum_{m=1}^4\|A_m\|_1).$$
\end{proof}

\section{Proof of Theorem \ref{abs lip}}\label{Sect=MainTheorem}

The following lemma can be found in \cite{Bhatia11}. Its short proof is included for convenience of the reader.

\begin{lem}\label{bhatia positive} If $\alpha_k>0,$ $0\leq k<n,$ are decreasing, then the matrix
$$\Phi=\left\{\frac{\alpha_{\max\{k,l\}}}{\alpha_k+\alpha_l}\right\}_{k,l=0}^{n-1}$$
is positive semidefinite.
\end{lem}
\begin{proof} Set
$$\Phi_1=\left\{\frac1{\alpha_k+\alpha_l}\right\}_{k,l=0}^{n-1}.$$
Consider (rank one) projections $p_k,$ $0\leq k<n,$ given by diagonal matrix units. That is,
$$(p_k)_{i,j}=
\begin{cases}
0,\quad i\neq k\\
0,\quad j\neq k\\
1,\quad i=j=k
\end{cases}
$$
and set
$$
P_k:=\sum _{j=0}^{k-1} p_j,\ 0\leq k<n.
$$
The following equality can be verified directly.\footnote{For $n=2,$ we have
$$\begin{pmatrix}
\frac{\alpha_0}{\alpha_0+\alpha_0}&\frac{\alpha_1}{\alpha_0+\alpha_1}\\
\frac{\alpha_1}{\alpha_0+\alpha_1}&\frac{\alpha_1}{\alpha_1+\alpha_1}
\end{pmatrix}
=
\begin{pmatrix}
\frac{\alpha_0-\alpha_1}{\alpha_0+\alpha_0}&0\\
0&0
\end{pmatrix}
+
\begin{pmatrix}
\frac{\alpha_1}{\alpha_0+\alpha_0}&\frac{\alpha_1}{\alpha_0+\alpha_1}\\
\frac{\alpha_1}{\alpha_0+\alpha_1}&\frac{\alpha_1}{\alpha_1+\alpha_1}
\end{pmatrix}.$$
For $n=3,$ we have
$$\begin{pmatrix}
\frac{\alpha_0}{\alpha_0+\alpha_0}&\frac{\alpha_1}{\alpha_0+\alpha_1}&\frac{\alpha_2}{\alpha_0+\alpha_2}\\
\frac{\alpha_1}{\alpha_0+\alpha_1}&\frac{\alpha_1}{\alpha_1+\alpha_1}&\frac{\alpha_2}{\alpha_1+\alpha_2}\\
\frac{\alpha_2}{\alpha_0+\alpha_2}&\frac{\alpha_2}{\alpha_1+\alpha_2}&\frac{\alpha_2}{\alpha_2+\alpha_2}
\end{pmatrix}
=$$
$$=
\begin{pmatrix}
\frac{\alpha_0-\alpha_1}{\alpha_0+\alpha_0}&0&0\\
0&0&0\\
0&0&0
\end{pmatrix}
+
\begin{pmatrix}
\frac{\alpha_1-\alpha_2}{\alpha_0+\alpha_0}&\frac{\alpha_1-\alpha_2}{\alpha_0+\alpha_1}&0\\
\frac{\alpha_1-\alpha_2}{\alpha_0+\alpha_1}&\frac{\alpha_1-\alpha_2}{\alpha_1+\alpha_1}&0\\
0&0&0
\end{pmatrix}
+\begin{pmatrix}
\frac{\alpha_2}{\alpha_0+\alpha_0}&\frac{\alpha_2}{\alpha_0+\alpha_1}&\frac{\alpha_2}{\alpha_0+\alpha_2}\\
\frac{\alpha_2}{\alpha_0+\alpha_1}&\frac{\alpha_2}{\alpha_1+\alpha_1}&\frac{\alpha_2}{\alpha_1+\alpha_2}\\
\frac{\alpha_2}{\alpha_0+\alpha_2}&\frac{\alpha_2}{\alpha_1+\alpha_2}&\frac{\alpha_2}{\alpha_2+\alpha_2}
\end{pmatrix}.$$
For larger $n,$ the decomposition follows exactly the same way.
}
\begin{equation}\label{baq}
\Phi=\Big(\sum_{k=0}^{n-2}(\alpha_k-\alpha_{k+1})P_k\Phi_1 P_k\Big)+\alpha_{n-1}\Phi_1.
\end{equation}
It is well known (see e.g. \cite{Bhatia}) that the Cauchy matrix $\Phi_1$ is positive semidefinite. It is now immediate from \eqref{baq} that the matrix $\Phi$ is also positive semidefinite.
\end{proof}

\begin{lem}\label{key lemma} Let $\alpha_k>0,$ $0\leq k<n,$ be decreasing. Define an operator $S:M_n\to M_n$ by setting
$$S(A)=\sum_{k,l=0}^{n-1}\frac{\alpha_k-\alpha_l}{\alpha_k+\alpha_l}p_kAp_l,$$
where $p_k$, $0\leq k<n,$ are the pairwise orthogonal rank one projections in $M_n$. We have
$$\|S(A)\|_{1,\infty}\leq\frac{80e}{\pi}\|A\|_1$$
for every $A\in M_n.$
\end{lem}
\begin{proof} It is sufficient to prove the assertion for the special case of projections $p_k$ defined in the proof of the preceding lemma.
Set $T$ to be the triangular truncation operator defined by setting
$$(T(A))_{i,j}=
\begin{cases}
A_{i,j},\quad i\geq j\\
0,\quad i<j
\end{cases}
$$
and let $M_{\Phi}$ be the Schur multiplication operator with respect to $\Phi$ from Lemma \ref{bhatia positive}. We have
$$S=(2T-1)(2M_{\Phi}-1).$$
Indeed,
$$(S(A))_{k,l}=\frac{\alpha_k-\alpha_l}{\alpha_k+\alpha_l}A_{k,l}=(\frac{2\alpha_k}{\alpha_k+\alpha_l}-1)A_{k,l}=((2M_{\Phi}-1)(A))_{k,l},\quad k\geq l$$
and
$$(S(A))_{k,l}=\frac{\alpha_k-\alpha_l}{\alpha_k+\alpha_l}A_{k,l}=(1-\frac{2\alpha_l}{\alpha_k+\alpha_l})A_{k,l}=-((2M_{\Phi}-1)(A))_{k,l},\quad k<l.$$
It is known (see \cite[Theorem IV.8.2]{GK1}) that
$$\|(2T-1)(X)\|_{1,\infty}\leq \frac{4e}{\pi}\|X\|_1,\quad X=X^*\in M_n,\ {\rm diag}(X)=0.$$
Thus,
$$\|(2T-1)(X)\|_{1,\infty}\leq\frac{16e}{\pi}\|X\|_1,\quad X\in M_n,\ {\rm diag}(X)=0.$$
By Lemma \ref{schur positive} and Lemma \ref{bhatia positive}, we have
$$\|2M_{\Phi}-1\|_{\mathcal{L}_1\to\mathcal{L}_1}\leq 1+8\|{\rm diag}(\Phi)\|_{\infty}=5.$$
Therefore,
$$\|S(A)\|_{1,\infty}\leq\frac{16e}{\pi}\|(2M_{\Phi}-1)(A)\|_1\leq\frac{80e}{\pi}\|A\|_1.$$
\end{proof}

\begin{lem}\label{prefinal lemma} Let $A,B\in M_{2n}$ be identically and symmetrically distributed matrices. We have
$$\||A|-|B|\|_{1,\infty}\leq\Big(8+\frac{640e}{\pi}\Big)\|A-B\|_1.$$
\end{lem}
\begin{proof} We have
$$A=\sum_{k=0}^{n-1}\mu(k,A_+)p_{1k}-\sum_{k=0}^{n-1}\mu(k,A_+)p_{2k},\quad |A|=\sum_{k=0}^{n-1}\mu(k,A_+)p_{1k}+\sum_{k=0}^{n-1}\mu(k,A_+)p_{2k},$$
$$B=\sum_{l=0}^{n-1}\mu(l,A_+)q_{1l}-\sum_{l=0}^{n-1}\mu(l,A_+)q_{2l},\quad |B|=\sum_{l=0}^{n-1}\mu(l,A_+)q_{1l}+\sum_{l=0}^{n-1}\mu(l,A_+)q_{2l},$$
where all the projections $p_{1k},$ $p_{2k},$ $0\leq k<n$ are pairwise orthogonal and have rank $1$ (and the same holds for the projections $q_{1l},$ $q_{2l},$ $0\leq l<n$). Hence, we have
\begin{align*}|A|-|B|&=\sum_{k,l=0}^{n-1}(\mu(k,A_+)-\mu(l,A_+))p_{1k}q_{1l}+\sum_{k,l=0}^{n-1}(\mu(k,A_+)-\mu(l,A_+))p_{2k}q_{2l}\cr &
+\sum_{k,l=0}^{n-1}(\mu(k,A_+)-\mu(l,A_+))p_{1k}q_{2l}+\sum_{k,l=0}^{n-1}(\mu(k,A_+)-\mu(l,A_+))p_{2k}q_{1l}\cr &
=\sum_{k,l=0}^{n-1}p_{1k}(A-B)q_{1l}-\sum_{k,l=0}^{n-1}p_{2k}(A-B)q_{2l}\cr &
+\sum_{k,l=0}^{n-1}\frac{\mu(k,A_+)-\mu(l,A_+)}{\mu(k,A_+)+\mu(l,A_+)}p_{1k}(A-B)q_{2l}\cr &
-\sum_{k,l=0}^{n-1}\frac{\mu(k,A_+)-\mu(l,A_+)}{\mu(k,A_+)+\mu(l,A_+)}p_{2k}(A-B)q_{1l}.
\end{align*}
Take unitary matrices $U,V\in M_{2n}$ such that $q_{2l}=Up_{1l}U^{-1}$ and $q_{1l}=Vp_{2l}V^{-1}$ for all $0\leq l<n.$ It is clear that
$$p_{1k}(A-B)q_{2l}=p_{1k}\Big({\rm supp}(A_+)(A-B)U{\rm supp}(A_+)\Big)p_{1l}\cdot U^{-1},\ 0\leq k, l<n$$
$$p_{2k}(A-B)q_{1l}=p_{2k}\Big({\rm supp}(A_-)(A-B)V{\rm supp}(A_-)\Big)p_{2l}\cdot V^{-1}\ 0\leq k, l<n.$$
Consider the operators $S_1:{\rm supp}(A_+)M_{2n}{\rm supp}(A_+)\to {\rm supp}(A_+)M_{2n}{\rm supp}(A_+)$
$$S_1(X):=\sum_{k,l=0}^{n-1}\frac{\mu(k,A_+)-\mu(l,A_+)}{\mu(k,A_+)+\mu(l,A_+)}p_{1k}Xp_{1l},\quad X\in {\rm supp}(A_+)M_{2n}{\rm supp}(A_+)$$
and $S_2:{\rm supp}(A_-)M_{2n}{\rm supp}(A_-)\to {\rm supp}(A_-)M_{2n}{\rm supp}(A_-)$
$$S_2(X):=\sum_{k,l=0}^{n-1}\frac{\mu(k,A_+)-\mu(l,A_+)}{\mu(k,A_+)+\mu(l,A_+)}p_{2k}Xp_{2l},\quad X\in {\rm supp}(A_-)M_{2n}{\rm supp}(A_-).$$
Employing these notations, we obtain
\begin{align*}|A|&-|B|={\rm supp}(A_+)(A-B){\rm supp}(B_+)-{\rm supp}(A_-)(A-B){\rm supp}(B_-)\cr &
+S_1({\rm supp}(A_+)(A-B)U{\rm supp}(A_+))\cdot U^{-1}-S_2({\rm supp}(A_-)(A-B)V{\rm supp}(A_-))\cdot V^{-1}.
\end{align*}
Since the algebras ${\rm supp}(A_+)M_{2n}{\rm supp}(A_+)$ and ${\rm supp}(A_-)M_{2n}{\rm supp}(A_-)$ are $*-$isomorphic to the algebra $M_n,$ it follows that the operators $S_1$ and $S_2$ satisfy the assumptions of Lemma \ref{key lemma}. Applying Lemma \ref{key lemma}, we obtain
\begin{align*}
 \frac14\||A|&-|B|\|_{1,\infty}\leq \|{\rm supp}(A_+)(A-B){\rm supp}(B_+)\|_{1,\infty}+
\|{\rm supp}(A_-)(A-B){\rm supp}(B_-)\|_{1,\infty}\cr &
+\|S_1({\rm supp}(A_+)(A-B)U{\rm supp}(A_+))\|_{1,\infty}
+\|S_2({\rm supp}(A_-)(A-B)V{\rm supp}(A_-))\|_{1,\infty},
\end{align*}
and
$$\||A|-|B|\|_{1,\infty}\leq (4+4+4\cdot \frac{80e}{\pi}+4\cdot\frac{80e}{\pi})\|A-B\|_1=\Big(8+\frac{640e}{\pi}\Big)\|A-B\|_1.$$
\end{proof}

In the following lemma, we get rid of the auxiliary conditions on $A$ and $B$ imposed in Lemma \ref{prefinal lemma}.

\begin{lem}\label{abs lip fin} For all self-adjoint matrices $A,B\in M_{2n},$ we have
$$\||A|-|B|\|_{1,\infty}\leq\Big(34+\frac{2560e}{\pi}\Big)\cdot\|A-B\|_1.$$
\end{lem}
\begin{proof} Let matrices $A,B$ be symmetrically (but not necessarily identically) distributed,
$$A=\sum_{k=0}^{n-1}\mu(k,A_+)p_{1k}-\sum_{k=0}^{n-1}\mu(k,A_+)p_{2k},\ B=\sum_{k=0}^{n-1}\mu(k,B_+)q_{1k}-\sum_{k=0}^{n-1}\mu(k,B_+)q_{2k},$$
where all the projections $p_{1k}$, $p_{2k}$, $q_{1k}$, $q_{2k}$, $0\leq k<n$ are pairwise orthogonal and have rank $1$. We introduce an auxiliary matrix
$$C:=\sum_{k=0}^{n-1}\mu(k,B_+)p_{1k}-\sum_{k=0}^{n-1}\mu(k,B_+)p_{2k}.$$
Clearly, $B$ and $C$ are identically and symmetrically distributed matrices (in particular, we have $\mu(B)=\mu(C)$). Thus, we have
$$\||B|-|A|\|_{1,\infty}=\|(|B|-|C|)+(|C|-|A|)\|_{1,\infty}\leq 2\||B|-|C|\|_{1,\infty}+2\||A|-|C|\|_{1,\infty}.$$
By Lemma \ref{prefinal lemma}, we have
\begin{align*}\||B|-|C|\|_{1,\infty}&\leq\Big(8+\frac{640e}{\pi}\Big)\cdot\|B-C\|_1\leq \Big(8+\frac{640e}{\pi}\Big)\cdot(\|A-B\|_1+\|A-C\|_1)\cr &
=\Big(8+\frac{640e}{\pi}\Big)\cdot(\|A-B\|_1+\|\mu(A)-\mu(B)\|_1)\cr & \leq \Big(16+\frac{1280e}{\pi}\Big)\|A-B\|_1.
\end{align*}
Here, we used the fact (guaranteed by our definition of $C$) that $\|A-C\|_1=\|\mu(A)-\mu(B)\|_1$. Since the matrices $A$ and $C$ commute, it follows that
$$\||A|-|C|\|_{1,\infty}\leq\|A-C\|_{1,\infty}\leq\|A-C\|_1=\|\mu(A)-\mu(B)\|_1\leq\|A-B\|_1,$$
where in the last step we used the classical fact \cite[(1.22)]{Simon}.
Combining the above inequalities we complete the proof for the case of symmetrically distributed matrices.

Let now $A$ and $B$ be arbitrary self-adjoint matrices from $M_{2n}$. Consider an element $F\in M_2$ given by
$$F:=
\begin{pmatrix}
1&0\\
0&-1
\end{pmatrix}
$$
and observe that
$$|A\otimes F|-|B\otimes F|=(|A|-|B|)\otimes 1,$$
where $1$ is the identity in $M_2.$ Note that
$$\|X\otimes 1\|_{1,\infty}=2\|X\|_{1,\infty},\quad \|X\otimes 1\|_1=2\|X\|_1$$
Now, observing that $A\otimes F$ and $B\otimes F$ are symmetrically distributed matrices, we infer from the first part of the proof that
\begin{align*}\||A|-|B|\|_{1,\infty}&=\frac12\||A\otimes F|-|B\otimes F|\|_{1,\infty}\cr &
\leq\frac12\cdot (34+\frac{2560e}{\pi}\Big)\|A\otimes F-B\otimes F\|_1\cr &=(34+\frac{2560e}{\pi}\Big)\cdot\|A-B\|_1.\end{align*}
\end{proof}

\begin{proof}[Proof of Theorem \ref{abs lip}] Let $p_n,$ $n\geq 0,$ be a sequence of finite rank projections in $\mathcal{L}(H)$ such that $p_n\uparrow 1.$ By \cite[Corollary 1.5]{DDPS1},
$$\Big(|p_nAp_n|-|p_nBp_n|\Big)E\to\Big(|A|-|B|\Big)E$$
uniformly for every finite rank projection $E.$ By Lemma \ref{abs lip fin}, we have
$$\||p_nAp_n|-|p_nBp_n|\|_{1,\infty}\leq \Big(34+\frac{2560e}{\pi}\Big)\|p_n(A-B)p_n\|_1\leq \Big(34+\frac{2560e}{\pi}\Big)\|A-B\|_1.$$
Thus,
$$\mu(k,(|p_nAp_n|-|p_nBp_n|)E)\leq\mu(k,|p_nAp_n|-|p_nBp_n|)\leq\Big(34+\frac{2560e}{\pi}\Big)\|A-B\|_1\cdot\frac1{k+1},\quad k\geq0,$$
for every finite rank projection $E.$ Since uniform convergence implies the convergence of singular values, it follows that
$$\mu(k,(|A|-|B|)E)\leq\Big(34+\frac{2560e}{\pi}\Big)\|A-B\|_1\cdot\frac1{k+1},\quad k\geq0,$$
for every finite rank projection $E.$ Hence, using judicious choice of projection $E$   (which is possible since our proof yields that $|A|-|B|$ is compact and hence $E$ may be taken to be a suitable spectral projection of this operator),  we have
$$\mu(k,|A|-|B|)\leq\Big(34+\frac{2560e}{\pi}\Big)\|A-B\|_1\cdot\frac1{k+1},\quad k\geq0.$$
\end{proof}

\section{Theorem \ref{abs lip} for compact operators}\label{Sect=CompactOperators}

Define an ideal $(\mathcal{L}_{1,\infty})_0$ in $\mathcal{L}(H)$ by setting
$$(\mathcal{L}_{1,\infty})_0=\{A\in\mathcal{K}(H):\quad k\mu(k,A)\to0\mbox{ as }k\to\infty\}.$$
This ideal coincides with the closure of the ideal of all finite rank operators in $\mathcal{L}_{1,\infty}$ and is commonly called the separable part of $\mathcal{L}_{1,\infty}.$

Define a (non-linear) functional $\theta$ on $\mathcal{L}_{1,\infty}$ by setting
$$\theta(A)=\limsup_{k\to\infty}k\mu(k,A).$$

\begin{lem}\label{theta lemma} Let $A,B\in\mathcal{L}_{1,\infty}.$ We have
\begin{enumerate}[{\rm (a)}]
\item\label{thetaa} If $\frac1\alpha+\frac1{\beta}=1,$ then
$$\theta(A+B)\leq \alpha\theta(A)+\beta\theta(B).$$
\item\label{thetab} If $B\in\mathcal{L}_1,$ then $\theta(B)=0.$
\item\label{thetac} If $B\in\mathcal{L}_1,$ then $\theta(A+B)=\theta(A).$
\end{enumerate}
\end{lem}
\begin{proof} We have
$$\mu(k,A+B)\leq \mu([\frac{k}{\alpha}],A)+\mu([\frac{k}{\beta}],B).$$
Hence,
$$\theta(A+B)=\limsup_{k\to\infty}k\mu(k,A+B)\leq\limsup_{k\to\infty}k\mu([\frac{k}{\alpha}],A)+\limsup_{k\to\infty}k\mu([\frac{k}{\beta}],B).$$
It is obvious that
$$\limsup_{k\to\infty}k\mu([\frac{k}{\alpha}],A)=\alpha\theta(A),\quad \limsup_{k\to\infty}k\mu([\frac{k}{\beta}],B)=\beta\theta(B).$$
This proves \eqref{thetaa}.

If $B\in\mathcal{L}_1$ and if $p_n\uparrow 1,$ then $\|B(1-p_n)\|_1\to0$ (see e.g. \cite{CS}). Let $e_k$ be the eigenvector of $|B|$ corresponding to the eigenvalue $\mu(k,B)$ and set $p_k$ to be the projection on the linear span of $e_n,$ $0\leq n<k.$ We have
$$\sum_{m=k}^{\infty}\mu(m,B)=\|B(1-p_n)\|_1\to0.$$
Therefore,
$$\frac{k}{2}\mu(k,B)\leq\sum_{n=[k/2]}^k\mu(k,B)\leq\sum_{n=[k/2]}^{\infty}\mu(k,B)\to0.$$
This proves \eqref{thetab}.

If $A\in\mathcal{L}_{1,\infty}$ and $B\in\mathcal{L}_1,$ then it follows from \eqref{thetaa} and \eqref{thetab} that $\theta(A+B)\leq\alpha\theta(A).$ Since $\alpha>1$ is arbitrary, it follows that $\theta(A+B)\leq\theta(A).$ Applying the same argument to the operators $A+B\in\mathcal{L}_{1,\infty}$ and $-B\in\mathcal{L}_1,$ we infer that $\theta(A)\leq\theta(A+B).$ This proves \eqref{thetac}.
\end{proof}

\begin{lem}\label{key lemma infinite} Let $\alpha_k>0,$ $k\geq0,$ be decreasing. Define an operator $S:\mathcal{L}_2\to\mathcal{L}_2$ by setting
$$S(A)=\sum_{k,l=0}^{\infty}\frac{\alpha_k-\alpha_l}{\alpha_k+\alpha_l}p_kAp_l,$$
where $p_k,$ $k\geq0,$ are the pairwise orthogonal rank one projections in $\mathcal{L}(H).$ We have $S(A)\in(\mathcal{L}_{1,\infty})_0$ and
$$\|S(A)\|_{1,\infty}\leq\frac{80e}{\pi}\|A\|_1$$
for every $A\in\mathcal{L}_1.$
\end{lem}
\begin{proof} The norm estimate can be proved in exactly the same way as in Lemma \ref{key lemma}. Set $P_n=\sum_{k=0}^{n-1}p_k.$ We have
$$A=(1-P_n)A(1-P_n)+(AP_n+P_nA(1-P_n))$$
and, therefore,
$$S(A)=S((1-P_n)A(1-P_n))+S(AP_n+P_nA(1-P_n)).$$
We have $S(AP_n+P_nA(1-P_n))\in\mathcal{L}_1$ since this operator has finite rank (even though its norm may be quite large). It follows from Lemma \ref{theta lemma} that
$$\theta(S(A))=\theta(S((1-P_n)A(1-P_n))) $$
$$\leq\|S((1-P_n)A(1-P_n))\|_{1,\infty}\leq{\rm const}\cdot\|(1-P_n)A(1-P_n)\|_1.$$
However, $\|(1-P_n)A(1-P_n)\|_1\to0$ as $n\to\infty.$ It follows that $\theta(S(A))=0$ and, therefore, $S(A)\in(\mathcal{L}_{1,\infty})_0.$
\end{proof}

\begin{thm}\label{Thm=SeperablePart} If $A,B\in\mathcal{L}(H)$ are compact operators such that $A-B\in\mathcal{L}_1,$ then $|A|-|B|\in (\mathcal{L}_{1,\infty})_0.$
\end{thm}
\begin{proof} The proof follows that in Lemma \ref{prefinal lemma} and Lemma \ref{abs lip fin} {\it mutatis mutandi}.

\end{proof}

\section{General semifinite version of Theorem \ref{abs lip}}\label{Sect=SemiFinite}

We begin by recalling a few relevant facts and notations from the theory of noncommutative integration on semifinite von Neumann algebras. For details on von Neumann algebra
theory, the reader is referred to e.g. \cite{Dix}, \cite{KR1}, \cite{KR2}
or \cite{Tak}. General facts concerning measurable operators may
be found in \cite{Ne}, \cite{Se} (see also \cite[Chapter
IX]{Ta2}). For the convenience of the reader, some of the basic
definitions are recalled.

In what follows, let $\mathcal{M}$ be a von Neumann algebra on a separable Hilbert space $H.$ A linear operator $A:{\rm dom}(A)\to H,$ where the domain ${\rm dom}(A)$ of $A$ is a linear subspace of $H,$ is said to be {\it affiliated} with $\mathcal{M}$ if it commutes with every element in $\mathcal{M}^{\prime}.$

An operator $A$ affiliated with $\mathcal{M}$ is called $\tau-$measurable if there exists a sequence $\left\{p_n\right\}_{n=1}^{\infty}$ of $\tau-$finite projections in $\mathcal{M}$ such that $p_n\downarrow 0$ and $(1-p_n)\left(H\right)\subset {\rm dom}(A)$ for all $n.$ The collection $\mathcal{S}(\mathcal{M},\tau) $ of all $\tau-$measurable operators is a unital $*-$algebra with respect to the strong sum and strong multiplication. It is well known that a linear operator $A$ affiliated with $\mathcal{M}$ belongs to $\mathcal{S}(\mathcal{M},\tau) $ if and only if there exists $\lambda>0$ such that $$\tau(E_{|A|}(\lambda,\infty))<\infty.$$
Here, $E_{|A|}$ is the spectral family of the operator $|A|$.
Alternatively, an unbounded operator $A$ affiliated with $\mathcal{M}$ is $\tau-$measurable (see \cite{FK}) if and only if
$$\tau\left(E_{|A|}(n,\infty)\right)=o(1),\quad n\to\infty.$$

Let a semifinite von Neumann  algebra $\mathcal{M}$ be equipped with a faithful normal semi-finite trace $\tau.$ Let $A\in\mathcal{S}(\mathcal{M},\tau).$ The generalized singular value function $\mu(A):t\rightarrow \mu(t;A)$ of the operator $A$ is defined by setting
$$\mu(s;A)=\inf\{\|A(1-p)\|_{\infty}:\ p\in\mathcal{M}\mbox{ is a projection,}\ \tau(p)\leq s\}.$$
There exists an equivalent definition which involves the distribution function of the operator $|A|.$ For every self-adjoint operator $A\in\mathcal{S}(\mathcal{M},\tau),$ setting
$$d_A(t)=\tau(E_A(t,\infty)),\quad t>0,$$
we have (see e.g. \cite{FK})
$$\mu(t;A)=\inf\{s\geq0:\ d_{|A|}(s)\leq t\}.$$

If $\mathcal{M}=\mathcal{L}(H)$ and $\tau$ is the standard trace ${\rm Tr},$ then it is not difficult to see that $\mathcal{S}(\mathcal{M},\tau)=\mathcal{M}.$ In this case, for $A\in\mathcal{M},$ we have
$$\mu(n;A)=\mu(t;A),\quad t\in[n,n+1),\quad n\geq0.$$
The sequence $\{\mu(n;A)\}_{n\geq0}$ is just the sequence of singular values of the operator $A.$

For every $\varepsilon,\delta>0,$ we define the set
$$V(\varepsilon,\delta)=\{x\in\mathcal{S}(\mathcal{M},\tau):\ \exists p=p^2=p^*\in\mathcal{M}\mbox{ such that } \|x(1-p)\|\leq\varepsilon,\ \tau(p)\leq\delta\}.$$
The topology generated by the sets $V(\varepsilon,\delta),$ $\varepsilon,\delta>0,$ is called a measure topology.

Let $L_1(0,\infty)$ and $L_{\infty}(0,\infty)$ be Lebesgue spaces on $(0,\infty).$

We define the space
$$\mathcal{L}_1(\mathcal{M},\tau)=\{A\in \mathcal{S}(\mathcal{M},\tau):\ \mu(A)\in L_1(0,\infty)\}.$$
It is well-known that the functional
$$\|\cdot\|_1:A\to \|\mu(A)\|_1$$
is a Banach norm on $\mathcal{L}_1(\mathcal{M},\tau).$ Similarly, we say that $A\in(\mathcal{L}_1+\mathcal{L}_{\infty})(\mathcal{M},\tau)$ if and only if
$\mu(A)\in(L_1+L_{\infty})(0,\infty).$ Here, we identify $\mathcal{M}$ with $\mathcal{L}_\infty(\mathcal{M},\tau).$ The space $(\mathcal{L}_1+\mathcal{L}_{\infty})(\mathcal{M},\tau)$ can be also viewed as a sum of Banach spaces $\mathcal{L}_1(\mathcal{M},\tau)$ and $\mathcal{L}_{\infty}(\mathcal{M},\tau)$ (the latter space is equipped with the uniform norm, which we denote simply by $\|\cdot\|_{\infty}$).

Define a linear space
$$(\mathcal{L}_1+\mathcal{L}_{\infty})(\mathcal{M},\tau)=\{A\in\mathcal{S}(\mathcal{M},\tau):\ \mu(A)\in L_1+L_{\infty}\}.$$

One can define the noncommutative weak $L_1$ space in a similar manner. Set
$$\mathcal{L}_{1,\infty}(\mathcal{M},\tau)=\{A\in\mathcal{S}(\mathcal{M},\tau):\ \sup_{t>0}t\mu(t,A)<\infty\}.$$
The mapping
$$\|\cdot\|_{1,\infty}:A\to \sup_{t>0}t\mu(t,A),\quad A\in\mathcal{L}_{1,\infty}(\mathcal{M},\tau)$$
defines a quasi-norm on $\mathcal{L}_{1,\infty}(\mathcal{M},\tau).$ It can be easily seen that $\mathcal{L}_{1,\infty}(\mathcal{M},\tau)$ equipped with the latter quasi-norm becomes a quasi-Banach space. The quasi-Banach space $\mathcal{L}_{1,\infty}(\mathcal{M},\tau)$ has the Fatou property: if $A_n\in\mathcal{L}_{1,\infty}(\mathcal{M},\tau),$ $\|A_n\|_{1,\infty}\leq1$ and $A_n\to A$ in measure, then $A\in\mathcal{L}_{1,\infty}(\mathcal{M},\tau)$ and $\|A\|_{1,\infty}\leq 1.$

\begin{lem}\label{key lemma semif} Let $\mathcal{M}$ be a semifinite von Neumann algebra. Let $\alpha_k>0,$ $0\leq k<n,$ be decreasing. Define an operator $S:\mathcal{M}\to\mathcal{M}$ by setting
$$S(A)=\sum_{k,l=0}^{n-1}\frac{\alpha_k-\alpha_l}{\alpha_k+\alpha_l}p_kAp_l,$$
where $p_k,$ $0\leq k<n,$ are the pairwise orthogonal $\tau-$finite projections in $\mathcal{M}.$ We have
$$\|S(A)\|_{1,\infty}\leq{\rm const}\cdot\|A\|_1,\quad A\in\mathcal{M}.$$
\end{lem}
\begin{proof} The proof follows that of Lemma \ref{key lemma} {\it mutatis mutandi}. The reference to \cite{GK1} must be replaced with the reference to Theorem 1.4 in \cite{DDPS2}.
\end{proof}

\begin{lem}\label{prefinal lemma semif} Let $\mathcal{M}$ be a semifinite factor. Let $A,B\in \mathcal{M}$ be identically and symmetrically distributed finitely supported operators. We have
$$\||A|-|B|\|_{1,\infty}\leq{\rm const}\|A-B\|_1.$$
\end{lem}
\begin{proof} Since the type I factors were already treated in Theorem \ref{abs lip}, we may assume without loss of generality that $\mathcal{M}$ is a type II factor. Suppose first that $\mu(A)$ (and, hence, $\mu(B)$) takes finitely many values. The proof in this case follows that of Lemma \ref{prefinal lemma} {\it mutatis mutandi}. Observe that in the last argument we have used the assumption that $\mathcal M$ is a factor   to find unitaries $U$ and $V$ as in the proof of Lemma \ref{prefinal lemma} (recall: any two projections in a type II factor with equal finite trace are unitarily equivalent \cite[Theorem V.1.8]{Tak}).

Let now $A,B\in\mathcal{M}$ be arbitrary identically and symmetrically distributed finitely supported operators. There exist projections $p_{k,s},$ $s>0,$ $1\leq k\leq 4,$ such that $\tau(p_{k,s})=s$ and
$$A_+=\int_0^{\infty}\mu(s,A_+)dp_{1,s},\quad A_-=\int_0^{\infty}\mu(s,A_+)dp_{2,s},$$
$$B_+=\int_0^{\infty}\mu(s,A_+)dp_{3,s},\quad B_-=\int_0^{\infty}\mu(s,A_+)dp_{4,s}.$$
Fix $\varepsilon$ such that
$$\int_0^{\varepsilon}\mu(s,A_+)ds\leq\|A-B\|_1$$
and set
$$C_+=\sum_{k=0}^{\infty}\mu((k+1)\varepsilon,A_+)\Big(p_{1,(k+1)\varepsilon}-p_{1,k\varepsilon}\Big),\quad C_-=\sum_{k=0}^{\infty}\mu((k+1)\varepsilon,A_+)\Big(p_{2,(k+1)\varepsilon}-p_{2,k\varepsilon}\Big),$$
$$D_+=\sum_{k=0}^{\infty}\mu((k+1)\varepsilon,A_+)\Big(p_{3,(k+1)\varepsilon}-p_{3,k\varepsilon}\Big),\quad D_-=\sum_{k=0}^{\infty}\mu((k+1)\varepsilon,A_+)\Big(p_{4,(k+1)\varepsilon}-p_{4,k\varepsilon}\Big).$$
It is easy to see that
$$\|A_+-C_+\|_1=\sum_{k=0}^{\infty}\int_{k\varepsilon}^{(k+1)\varepsilon}\Big(\mu(s,A_+)-\mu((k+1)\varepsilon,A_+)\Big)ds$$
$$\leq \int_0^{\varepsilon}\Big(\mu(s,A_+)-\mu(\varepsilon,A_+)\Big)ds+\sum_{k=1}^{\infty}\varepsilon(\mu(k\varepsilon,A_+)-\mu((k+1)\varepsilon,A_+)\Big) $$
$$=\int_0^{\varepsilon}\Big(\mu(s,A_+)-\mu(\varepsilon,A_+)\Big)ds+\varepsilon\mu(\varepsilon,A_+)=\int_0^{\varepsilon}\mu(s,A_+)ds\leq\|A-B\|_1.$$

We have
$$\||A|-|B|\|_{1,\infty}=\|(|A|-|C|)+(|C|-|D|)-(|B|-|D|)\|_{1,\infty}\leq $$
$$\leq 4\||A|-|C|\|_{1,\infty}+4\||C|-|D|\|_{1,\infty}+4\||B|-|D|\|_{1,\infty}.$$
Recall that $C$ and $D$ are symmetrically and identically distributed finitely supported operators. By construction, $\mu(C)$ (and, hence, $\mu(D)$) takes only finitely many values. We infer from the previous paragraph that
$$\||C|-|D|\|_{1,\infty}\leq{\rm const}\cdot\|C-D\|_1.$$
Since $A$ and $C$ commute, it follows that
$$\||A|-|C|\|_{1,\infty}\leq\|A-C\|_{1,\infty}\leq\|A-C\|_1\leq 2\|A-B\|_1.$$
Similarly,
$$\||B|-|D|\|_{1,\infty}\leq 2\|A-B\|_1.$$
Also, we have
$$\|C-D\|_1=\|(C-A)+(A-B)+(B-D)\|_1\leq$$
$$\leq\|C-A\|_1+\|A-B\|_1+\|B-D\|_1\leq 5\|A-B\|_1.$$
Combining these estimates, we conclude the proof.
\end{proof}

The following lemma should be compared to the results on positive Schur multipliers in Section \ref{Sect=SchurMultiplication}.

\begin{lem}\label{Lem-NCSchur}
Let $\mathcal{M} \subseteq \mathcal{L}(H)$ be a von Neumann algebra. Let $B \in M_n$ and $B\geq0$. Let $p_1, \ldots, p_n$ be  mutually orthogonal projections in $\mathcal{M}$. Consider the operator valued Schur multiplier (or double operator integral) defined by,
\begin{equation}\label{Eqn=SB}
S_B: \mathcal{M} \rightarrow \mathcal{M}: x \mapsto \sum_{i,j=1}^n B_{i,j} p_i x p_j.
\end{equation}
Then $S_B$ preserves positive operators: $S_B (x) \geq 0$ whenever $x \geq 0$.
\end{lem}
\begin{proof}
The Schur multiplier extends to a map $S_B: \mathcal{L}(H) \rightarrow \mathcal{L}(H)$ prescribed by the same formula \eqref{Eqn=SB} and hence it suffices to prove the statement for $\mathcal{M} = \mathcal{L}(H)$. In case each of the projections $p_i$ are finite rank the statement is reduced to the matricial case and hence follows from Schur's theorem, see Section \ref{Sect=SchurMultiplication}. Indeed, this is true in case each $p_i$ is one dimensional. Else, write $p_i = \sum_m^{n_i} p_{i,m}$, a finite sum of mutually orthogonal rank 1 projections and apply the previous line to the set $\{ p_{i,m} \mid i, 1 \leq m \leq n_i\}$ using that $(B_{i,j})_{(i,1 \leq m \leq n_i),(j, 1 \leq k \leq n_j)}$ is again positive. The positivity of the latter matrix follows as this matrix is a corner of the Kronecker product $C_n\otimes B$ where $C_n$ is the $n\times n$-matrix with entries equal to 1. In the general case of not necessarily finite rank projections $p_i$ one can write each $p_i$ as a strong limit of finite rank projections $p_{i,m} \rightarrow p_i$. Putting $P_m = \sum_i p_{i,m}$ we see that $x \mapsto P_m S_B(x) P_m$ preserves positive operators and $P_m S_B(x) P_m \rightarrow S_B(x)$ strongly. This concludes the lemma.
\end{proof}

\begin{lem}\label{abs lip semif} Let $\mathcal{M}$ be a semifinite factor. Let $A,B\in\mathcal{M}$ be self-adjoint finitely supported operators. We have
$$\||A|-|B|\|_{1,\infty}\leq{\rm const}\|A-B\|_1.$$
\end{lem}
\begin{proof} The proof follows that of Lemma \ref{abs lip fin} {\it mutatis mutandi}. At the point that Schur's theorem is used, see the proof of Lemma \ref{schur positive}, Lemma \ref{Lem-NCSchur} can be invoked.
\end{proof}

\begin{lem}\label{final semif} If $\mathcal{M}$ is a semifinite factor and if $A,B\in(\mathcal{L}_1+\mathcal{L}_{\infty})(\mathcal{M},\tau)$ are such that $A-B\in\mathcal{L}_1(\mathcal{M},\tau),$ then $|A|-|B|\in\mathcal{L}_{1,\infty}(\mathcal{M},\tau)$ and
$$\||A|-|B|\|_{1,\infty}\leq{\rm const}\cdot\|A-B\|_1.$$
\end{lem}
\begin{proof} Suppose first that $A,B\in\mathcal{M}.$ Let $p_n,$ $n\geq 0,$ be a sequence of $\tau-$finite projections in $\mathcal{M}$ such that $p_n\uparrow 1.$ By \cite[Corollary 1.5]{DDPS1},
$$\Big(|p_nAp_n|-|p_nBp_n|\Big)E\to\Big(|A|-|B|\Big)E$$
in measure for every $\tau-$finite projection $E.$ By Lemma \ref{abs lip semif}, we have
$$\||p_nAp_n|-|p_nBp_n|\|_{1,\infty}\leq {\rm const}\cdot\|p_n(A-B)p_n\|_1\leq{\rm const}\cdot\|A-B\|_1.$$
Therefore,
$$\mu(t,(|p_nAp_n|-|p_nBp_n|)E)\leq\mu(t,|p_nAp_n|-|p_nBp_n|)\leq\frac{{\rm const}}{t}\|A-B\|_1,\quad t>0,$$
for every $\tau-$finite projection $E.$ Since convergence in measure implies the (almost everywhere) convergence of singular value functions (see e.g. \cite[Lemma 7]{Sbik}), it follows that
$$\mu(t,(|A|-|B|)E)\leq\frac{{\rm const}}{t}\|A-B\|_1,\quad t>0,$$
for every $\tau-$finite projection $E.$ Hence, using a judicious choice of the projection $E$ (namely a suitable spectral projection of $|A|-|B|$), we have
$$\mu(t,|A|-|B|)\leq\frac{{\rm const}}{t}\|A-B\|_1,\quad t>0.$$
This proves the assertion for bounded $A$ and $B.$

Let now $A,B\in(\mathcal{L}_1+\mathcal{L}_{\infty})(\mathcal{M},\tau).$ Set
$$p_n=E_{|A|}[0,n]\bigwedge E_{|B|}[0,n].$$
Since $A,B$ are $\tau-$measurable, it follows from \cite{Tikhonov} (see also \cite[Theorem 1.1]{DDPS1}) that
$$p_nAp_n\to A,\quad p_nBp_n\to B,\quad |p_nAp_n|\to|A|,\quad|p_nBp_n|\to |B|$$
in measure. It follows from the above that
$$\||p_nAp_n|-|p_nBp_n|\|_{1,\infty}\leq{\rm const}\cdot\|p_n(A-B)p_n\|_1\leq{\rm const}\cdot\|A-B\|_1.$$
Since the quasi-norm in $\mathcal{L}_{1,\infty}(\mathcal{M},\tau)$ has the Fatou property, it follows that
$$\||A|-|B|\|_{1,\infty}\leq{\rm const}\cdot\|A-B\|_1.$$
\end{proof}

The following lemma shows the proper triangle inequality in $\mathcal{L}_{1,\infty}$ for pairwise orthogonal summands. Let
$A_k\in\mathcal{L}_{1,\infty}(\mathcal{M},\tau),$ $k\geq0$. We use the direct sum symbol $\bigoplus_{k=0}^{\infty}A_k$ to denote the operator on $H$ formed with respect
to some arbitrary Hilbert space isomorphism $\bigoplus_{k=0}^{\infty}H\simeq H$.

\begin{lem}\label{disjoint triangle} If $A_k\in\mathcal{L}_{1,\infty}(\mathcal{M},\tau),$ $k\geq0,$ then
$$\|\bigoplus_{k=0}^{\infty}A_k\|_{1,\infty}\leq\sum_{k=0}^{\infty}\|A_k\|_{1,\infty}.$$
\end{lem}
\begin{proof} Set $x(t)=1/t,$ $t>0.$ For simplicity of notations, denote $\|A_k\|_{1,\infty}$ by $\alpha_k.$ We have $\mu(t,A_k)\leq \alpha_k/t,$ $t>0.$ Using the notation $d_{\alpha_kx}(t)$ for the classical distribution, it is immediate that
$$d_{|\bigoplus_{k=0}^{\infty}A_k|}(t)=\sum_{k=0}^{\infty}d_{|A_k|}(t)\leq\sum_{k=0}^{\infty}d_{\alpha_kx}(t)=\sum_{k=0}^{\infty}\frac{\alpha_k}{t}=\frac{\sum_{k=0}^{\infty}\alpha_k}{t}=d_{(\sum_{k=0}^{\infty}\alpha_k)x}(t).$$
Hence,
$$\mu(\bigoplus_{k=0}^{\infty}A_k)\leq(\sum_{k=0}^{\infty}\alpha_k)x$$
or, equivalently,
$$\|\bigoplus_{k=0}^{\infty}A_k\|_{1,\infty}\leq(\sum_{k=0}^{\infty}\alpha_k)\|x\|_{1,\infty}=\sum_{k=0}^{\infty}\|A_k\|_{1,\infty}.$$
\end{proof}

The following lemma combines well-known facts from \cite{Tak} and less known facts from \cite{Dykema}.

\begin{lem}\label{factor lemma} For every semifinite von Neumann algebra $(\mathcal{M},\tau),$ there exist semifinite factors $(\mathcal{M}_k,\tau_k),$ $k\geq0,$ and a trace preserving $^*-$monomorphism of $(\mathcal{M},\tau)$ into $(\bigoplus_{k\geq0}\mathcal{M}_k,\bigoplus_{k\geq0}\tau_k).$
\end{lem}
\begin{proof} By Theorem V.1.19 in \cite{Tak}, we have $\mathcal{M}=\mathcal{M}^1\oplus\mathcal{M}^2\oplus\mathcal{M}^3,$ where $\mathcal{M}^1$ is type I, $\mathcal{M}^2$ is type II$_1$ and $\mathcal{M}^3$ is type II$_{\infty}.$

By Theorem V.1.27 in \cite{Tak}, there exist commutative algebras $\mathcal{A}_k,$ $k\geq0,$ and Hilbert spaces $H_k,$ $k\geq0,$ such that
$$\mathcal{M}^1=\bigoplus_{k\geq0}\mathcal{A}_k\bar{\otimes}\mathcal{L}(H_k).$$
Every $\mathcal{A}_k$ admits a trace preserving isomorphic embedding into $L_{\infty}(0,\infty)$ and then to the hyperfinite II$_{\infty}$ factor $\mathcal{R}\bar{\otimes}\mathcal{L}(H).$ Every $\mathcal{L}(H_k)$ admits a trace preserving isomorphic embedding into $\mathcal{L}(H).$ Thus, $\mathcal{M}^1$ admits a trace preserving isomorphic embedding into
$$\bigoplus_{k\geq0}\mathcal{R}\bar{\otimes}\mathcal{L}(H)\bar{\otimes}\mathcal{L}(H).$$
Equivalently, $\mathcal{M}^1$ admits a trace preserving isomorphic embedding into $(\mathcal{R}\bar{\otimes}\mathcal{L}(H))^{\oplus\infty}.$

By Theorem 2.3 and Lemma 2.5 in \cite{Dykema}, the algebra $\mathcal{M}^2$ admits a trace preserving isomorphic embedding into a II$_1$ factor.

By Theorem V.1.40 in \cite{Tak}, there exist type II$_1-$algebras $\mathcal{N}_k,$ $k\geq0,$ such that
$$\mathcal{M}^3=\bigoplus_{k\geq0}\mathcal{N}_k\bar{\otimes}\mathcal{L}(H).$$
Again applying Theorem 2.3 and Lemma 2.5 in \cite{Dykema}, we embed the algebras $\mathcal{N}_k,$ $k\geq0,$ into II$_1$ factors $\mathcal{O}_k,$ $k\geq0.$ Since $\mathcal{O}_k\bar{\otimes}\mathcal{L}(H)$ is a type II$_{\infty}$ factor, the assertion follows for $\mathcal{M}^3$ and, thus, for $\mathcal{M}.$
\end{proof}

\begin{thm}\label{semifinite} If $\mathcal{M}$ is a semifinite von Neumann algebra and if $A,B\in(\mathcal{L}_1+\mathcal{L}_{\infty})(\mathcal{M},\tau)$ are such that $A-B\in\mathcal{L}_1(\mathcal{M},\tau),$ then $|A|-|B|\in\mathcal{L}_{1,\infty}(\mathcal{M},\tau)$ and
$$\||A|-|B|\|_{1,\infty}\leq{\rm const}\cdot\|A-B\|_1.$$
\end{thm}
\begin{proof} According to the Lemma \ref{factor lemma}, we can embed $(\mathcal{M},\tau)$ into $(\bigoplus_{k\geq0}\mathcal{M}_k,\bigoplus_{k\geq0}\tau_k),$ where $(\mathcal{M}_k,\tau_k),$ $k\geq0,$ are semifinite factors.

Thus, $A=\bigoplus_{k\geq0}A_k$ and $B=\bigoplus_{k\geq0}B_k,$ where $A_k,B_k\in(\mathcal{L}_1+\mathcal{L}_{\infty})(\mathcal{M}_k,\tau_k),$ $k\geq0.$ By Lemma \ref{disjoint triangle}, we have that
$$\||A|-|B|\|_{1,\infty}=\|\bigoplus_{k\geq0}|A_k|-|B_k|\|_{1,\infty}\leq\sum_{k\geq0}\||A_k|-|B_k|\|_{1,\infty}.$$
Since every $\mathcal{M}_k$ is a factor, it follows from Lemma \ref{final semif} that
$$\||A_k|-|B_k|\|_{1,\infty}\leq{\rm const}\cdot\|A_k-B_k\|_1.$$
Therefore, we have
$$\||A|-|B|\|_{1,\infty}\leq{\rm const}\cdot\sum_{k\geq0}\|A_k-B_k\|_1={\rm const}\cdot\|\bigoplus_{k\geq0}A_k-B_k\|_1={\rm const}\cdot\|A-B\|_1.$$
\end{proof}

\section{Commutator estimates}\label{Sect=CommutatorEstimates}

The proof of the following consequence is essentially the same as the implication (i) $\Rightarrow$ (ii) of \cite[Theorem 2.2]{DDPS1}. For completeness and the fact that \cite[Theorem 2.2]{DDPS1} is not directly applicable since we are dealing with estimates between different spaces (and one of them only has a quasi-norm) we have included the proof.

\begin{thm}\label{Thm=WkCommutatorEstimate} If $A,B\in(\mathcal{L}_1+\mathcal{L}_{\infty})(\mathcal{M},\tau)$ are self-adjoint operators such that $[A,B]\in\mathcal{L}_1(\mathcal{M},\tau),$ then
$$\|[|A|,B]\|_{1,\infty}\leq{\rm const}\cdot\|[A,B]\|_1.$$
\end{thm}
\begin{proof} Suppose first that $B$ is bounded. Setting $C=e^{i\varepsilon B}Ae^{-i\varepsilon B},$ we have $|C|=e^{i\varepsilon B}|A|e^{-i\varepsilon B}.$ We infer from Theorem \ref{semifinite} that
$$\|[e^{i\varepsilon B},|A|]\|_{1,\infty}=\||C|-|A|\|_{1,\infty}\leq{\rm const}\cdot\|C-A\|_1={\rm const}\cdot\|[e^{i\varepsilon B},A]\|_1.$$
However,
$$[e^{i\varepsilon B},A]=\sum_{k=1}^{\infty}\frac{(i\varepsilon)^k}{k!}[B^k,A]$$
where the series on the right hand side converges in $\mathcal{L}_1(\mathcal{M},\tau)$. Indeed,
$$[B^k,A]=\sum_{m=0}^{k-1}B^m[B,A]B^{k-1-m},\quad k\geq 1$$
and therefore,
$$\|[e^{i\varepsilon B},A]\|_1=\|\sum_{k=1}^{\infty}\sum_{m=0}^{k-1}\frac{(i\varepsilon)^k}{k!}B^m[B,A]B^{k-1-m}\|_1\leq\sum_{k=1}^{\infty}\sum_{m=0}^{k-1}\frac{\varepsilon^k}{k!}\|B^m[B,A]B^{k-1-m}\|_1\leq$$
$$\leq\sum_{k=1}^{\infty}\sum_{m=0}^{k-1}\frac{\varepsilon^k}{k!}\|B\|_{\infty}^{k-1}\|[B,A]\|_1=\Big(\sum_{k=1}^{\infty}\frac{k\varepsilon^k}{k!}\|B\|_{\infty}^{k-1}\Big)\|[A,B]\|_1=\varepsilon e^{\varepsilon\|B\|_{\infty}}\|[A,B]\|_1.$$
Combining preceding estimates, we infer that
$$\|\varepsilon^{-1}[e^{i\varepsilon B},|A|]\|_{1,\infty}\leq {\rm const}\cdot e^{\varepsilon\|B\|_{\infty}}\|[A,B]\|_1.$$
It follows from the Spectral Theorem and boundedness of $B$ that
$$\varepsilon^{-1}(e^{i\varepsilon B}-1)\to iB$$
uniformly and, therefore,
$$\varepsilon^{-1}[e^{i\varepsilon B},|A|]=[\varepsilon^{-1}(e^{i\varepsilon B}-1),|A|]\to i[B,|A|]$$
in the norm of the space $(\mathcal{L}_1+\mathcal{L}_{\infty})(\mathcal{M},\tau)$ and therefore
in measure (see also \cite{DDPS1}). Since the quasi-norm in $\mathcal{L}_{1,\infty}$ has the Fatou property, it follows that
$$\|[|A|,B]\|_{1,\infty}\leq {\rm const}\cdot\|[A,B]\|_1.$$
This proves the assertion for the case of bounded $B.$

Consider now the general case of an arbitrary $B\in(\mathcal{L}_1+\mathcal{L}_{\infty})(\mathcal{M},\tau).$ Set $p_n=E_{|B|}[0,n].$ We have that $p_nAp_n\to A$ and $p_nBp_n\to B$ in measure. Since $p_n$ commutes with $B,$ it follows that
$$[p_nAp_n,p_nBp_n]=p_n[A,B]p_n.$$
It follows from \cite{Tikhonov} (see also \cite[Theorem 1.1]{DDPS1}) that $|p_nAp_n|\to|A|$ in measure. Thus,
$$[|p_nAp_n|,p_nBp_n]\to[|A|,B]$$
in measure. It is proved in the previous paragraph that
$$\|[|p_nAp_n|,p_nBp_n]\|_{1,\infty}\leq{\rm const}\|[p_nAp_n,p_nBp_n]\|_1={\rm const}\cdot\|p_n[A,B]p_n\|_1\leq{\rm const}\cdot\|[A,B]\|_1.$$
Since the quasi-norm in $\mathcal{L}_{1,\infty}$ has the Fatou property, it follows that
$$\|[|A|,B]\|_{1,\infty}\leq {\rm const}\cdot\|[A,B]\|_1.$$
\end{proof}

\begin{rmk}\label{Rmk=WeakInterpolation}
The result of Theorem \ref{Daviesthm}  may be obtained from Theorem \ref{Thm=WkCommutatorEstimate} as follows. Firstly, observe that we can interpolate between the weak $\mathcal{L}_1$-space, $\mathcal{ L}_{1,\infty}$  and $\mathcal{L}_{2}$ using weak type interpolation (see e.g. \cite{Dirksen} and references therein). This immediately implies the estimates for Schatten $p$-norms, $1<p<2$ analogous to that of Theorem \ref{Thm=WkCommutatorEstimate}, with the case $2<p<\infty$ following by duality. The result of Theorem \ref{abs lip}  follows now from  \cite[Theorem 2.2]{DDPS1}.
\end{rmk}

\section{Final comments}

Davies introduced the class of functions representable in the form
$$f(t)=\int_{\mathbb{R}}|t-s|d\nu_f(s),$$
where $\nu_f$ is a signed measure with finite support. He proved that
$$\|f(A)-f(B)\|_p\leq c_{p,f}\|A-B\|_p,\quad 1<p<\infty.$$
Though we cannot fully extend this result to $p=1,$ the following is possible.

Define distorted variation $DV(\nu)$ as follows
$$DV(\nu)=\sup\{\inf_{\pi}\sum_{k\geq0}2^{\pi(k)}|\nu(A_k)|:\ A_m\cap A_n=\varnothing\mbox{ for all }m\neq n,\ \cup_{k\geq0}A_k=\mathbb{R}\}.$$
Here, every $A_n,$ $n\geq0,$ is an interval (or a semi-axis) and the infimum is taken over all permutations $\pi$ of $\mathbb{Z}_+$.

\begin{lem}\label{approx lemma} For every finitely supported measure $\nu$ with $DV(\nu)<\infty,$ there exists a sequence $\nu_m,$ $m\geq 1,$ of discrete measures such that
$$\int_{\mathbb{R}}|t-s|d\nu_m(s)\to\int_{\mathbb{R}}|t-s|d\nu(s)$$
uniformly and $DV(\nu_m)\leq DV(\nu)$ for all $m\ge 1$.
\end{lem}
\begin{proof} Assume, for simplicity of notations, that $\nu$ is supported on $[0,1)$ and that $|\nu|([0,1))=1.$ Define a measure $\nu_m$ by setting
$$\nu_m=\sum_{k=0}^{m-1}\nu([\frac{k}{m},\frac{k+1}{m}))\delta_{\{\frac{k}{m}\}}.$$
It is immediate that $DV(\nu_m)\leq DV(\nu)$ (because every partition of the finite set $\{0, 1/m, \cdots, (m-1)/m\}$ extends to a partition of $[0,1)$).

Fix $t\in \mathbb{R}$ and for a given $m\in\mathbb{N},$ define the function $g_m$ on $[0,1)$ by setting $g_m(s)=|k/m-t|$ for all $s\in[k/m,(k+1)/m),$ $0\leq k<m.$ Since $g_m$ is a step function with steps at $\{0, 1/m, \cdots, (m-1)/m\}$, it follows that
$$\int_{\mathbb{R}}g_m(s)d\nu_m(s)=\int_{\mathbb{R}}g_m(s)d\nu(s).$$
It is clear that
$$\Big|\int_{\mathbb{R}}|t-s|d\nu(s)-\int_{\mathbb{R}}g_m(s)d\nu(s)\Big|\leq\frac1m\cdot|\nu|([0,1))$$
and
$$\Big|\int_{\mathbb{R}}|t-s|d\nu_m(s)-\int_{\mathbb{R}}g_m(s)d\nu_m(s)\Big|\leq\frac1m\cdot|\nu_m|([0,1)).$$
It follows that
$$\Big|\int_{\mathbb{R}}|t-s|d\nu_m(s)-\int_{\mathbb{R}}|t-s|d\nu(s)\Big|\leq\frac{2}{m}.$$
This proves the claim.
\end{proof}

The following lemma is a particular case of  \cite[Lemma 17]{Suk-kv-ban} (proved there for every quasi-Banach space and not just $\mathcal{L}_{1,\infty}$). It serves as a replacement for the triangle inequality.

\begin{lem}\label{infinite quasi-traingle} Let $A_k\in\mathcal{L}_{1,\infty},$ $k\geq0.$ We have
$$\|\sum_{k=0}^{\infty}A_k\|_{1,\infty}\leq{\rm const}\cdot\sum_{k=0}^{\infty}2^k\|A_k\|_{1,\infty}.$$
Here, the convergence of the series in the right hand side guarantees that the series in the left hand side converges in $\mathcal{L}_{1,\infty}.$
\end{lem}

\begin{thm} If $f$ is in the Davies class and if $DV(\nu_f)<\infty,$ then for any two bounded self-adjoint operators $A$ and $B$, such that $A-B\in \mathcal{L}_1$, we have
$$\|f(A)-f(B)\|_{1,\infty}\leq{\rm const}\cdot DV(\nu_f)\cdot\|A-B\|_1.$$
\end{thm}
\begin{proof}   By Lemma \ref{approx lemma}, we may approximate $\nu$ with $DV(\nu)\leq 1$ by the sequence of discrete measures $\nu_m$ with $DV(\nu_m)\leq 1.$ It follows that we can assume without loss of generality that $\nu$ is discrete. Indeed, since $A$ and $B$ are bounded, we find that $f_m(A)\rightarrow f(A)$ and $f_m(B)\rightarrow f(B)$ uniformly, where
$$
f_m(t):= \int _{\mathbb{R}} |t-s|d\nu_m(s),\quad t\in \mathbb{R}.
$$ Using the Fatou property of the $\mathcal{L}_{1,\infty},$ we infer that it is suffices to prove the assertion for discrete measures with finite distorted variation.

If the measure $\nu$ is discrete, then
$$f(t)=\sum_{k=0}^{\infty}\alpha_k|t-t_k|,\quad\sum_{k=0}^{\infty}2^k|\alpha_k|<\infty.$$
We have
$$f(A)-f(B)=\sum_{k=0}^{\infty}\alpha_k\Big(|A-t_k|-|B-t_k|\Big).$$
By Theorem \ref{abs lip}, we have
$$\||A-t_k|-|B-t_k|\|_{1,\infty}\leq{\rm const}\cdot \|A-B\|_1.$$
It follows from Lemma \ref{infinite quasi-traingle} above that
$$\|f(A)-f(B)\|_{1,\infty}\leq \sum_{k=0}^{\infty}2^k|\alpha_k|\|A-B\|_1.$$
\end{proof}

\end{document}